\documentclass[11 pt]{amsart}

\usepackage{times}
\usepackage{geometry}
\usepackage{amssymb}
\usepackage{latexsym,amssymb,  amsmath, amscd, amsfonts}
\usepackage{graphicx}
\usepackage[percent]{overpic}
\usepackage{pdfsync}
\usepackage{units}
\usepackage{hyperref}
\usepackage{euscript}
\usepackage{multicol}
\usepackage{epstopdf}

\newtheorem{theorem}{Theorem}
\newtheorem{lemma}[theorem]{Lemma}
\newtheorem{proposition}[theorem]{Proposition}

\theoremstyle{definition}
\newtheorem{definition}[theorem]{Definition}

\newtheorem{conjecture}[theorem]{Conjecture}

\numberwithin{equation}{section}

\newcommand{\R}{\mathbb{R}}    
\newcommand{\bdry}{\partial}
\newcommand{\setm}{\smallsetminus}
\newcommand{\arc}[1]{\gamma_{#1}}
\newcommand{\len}[1]{\ell_{#1}}

\newcommand{\altval}{15.66}

\begin{document}

\title[Quadrisecants \& essential secants]{Quadrisecants and essential secants of knots: \\ with applications to the geometry of knots}

\author{Elizabeth Denne}
\address{Elizabeth Denne, Washington \& Lee University, Department of Mathematics, Lexington VA 24450, USA}
\email{dennee@wlu.edu}

\date{August 8, 2016.}

\begin{abstract}A quadrisecant line is one which intersects a curve in at least four points, while an essential secant captures something about the knottedness of a knot. This survey article gives a brief history of these ideas, and shows how they may be applied to questions about the geometry of a knot via the total curvature, ropelength and distortion of a knot.
\end{abstract}

%%%%%%%%%%%%%%%%%%%%%%%%%%%%%
\maketitle

%%%%%%%%%%%%%%%%%%%%%%%%%%%%%
\section{Introduction} \label{section:intro}

%Throughout this chapter, we will use \emph{knotted curve} to mean an oriented nontrivial tame knot in $\mathbb{R}^3$.

In this survey article, we consider the ways in which lines intersect knots. Some of these capture the knottedness of a knot. Trisecant and quadrisecant lines
are straight lines which intersect a knot in at least three, respectively four distinct places. It is clear that any closed curve has a $2$-parameter family secants. A simple dimension count shows there is a $1$-parameter family of trisecants and that quadrisecants are discrete ($0$-parameter family). A planar circle unknot does not have any trisecants and quadrisecants, however nontrivial tame knots must have them. In general, we do not expect knots  to have quintisecants (or higher order secants), they exist only for a codimension~1 (or higher) set of knots.  

The existence of quadrisecants and essential secants gives insight into the geometry of knots.  Section~\ref{section:quad} gives definitions of quadrisecants and essential secants and a brief history of known results. It ends with a discussion of the open question of finding bounds on the number of quadrisecants for a given knot type. Section~\ref{section:key-ideas} gives a brief outline of the tools used in many of the proofs of results about quadrisecants. Section~\ref{section:applications} shows how the ideas may be applied to results about the geometry of knots via the total curvature, second hull, ropelength and distortion of a knot.

%%%%%%%%%%%%%%%%%%%%%%%
\section{Quadrisecants} \label{section:quad}
Recall that a {\em knot} is a homeomorphic image of $S^1$ in $\R^3$, modulo reparametrizations, and a tame knot is one that is ambient isotopic to a polygonal knot. 

\begin{definition} Given a knot $K$, an {\em $n$-secant line} is an oriented line which intersects $K$ in at least $n$ components. An {\em $n$-secant} is an ordered $n$-tuple of points of $K$ (no two of which lie in a common straight subarc of $K$) which lie in order along an $n$-secant line.
\end{definition}

As previously described, by a secant we mean a $2$-secant, by a trisecant a $3$-secant, and by a quadrisecant a $4$-secant. For a closed curve~$K$, any two distinct
points determine a straight line. These points are a secant if and
only if they do not lie on a
common straight subarc of~$K$. Thus the set of
secants $S=K^2\setm\tilde{\Delta}$ and is 
topologically an annulus. (Here $\tilde{\Delta}$ denotes the set of $n$-tuples in which some pair of points lie in a common straight subarc of~$K$.)

We now consider the set of trisecants of a knot, denoted $\mathcal{T}\subset K^3\setm\tilde{\Delta}$.  For any trisecant $abc$, there are~$\vert S_3/C_3\vert=2$ cyclic orderings of the oriented knot and trisecant. We could label these by their lexicographically least elements ($abc$ or~$acb$), but we call them {\em direct} and {\em
reversed} respectively. Figure~\ref{fig:trisecants} illustrates the two types of
trisecant. Flipping the orientation of the knot or the trisecant would change its type.

\begin{figure}[htbp]
\begin{center}
\begin{overpic}[scale=0.9]{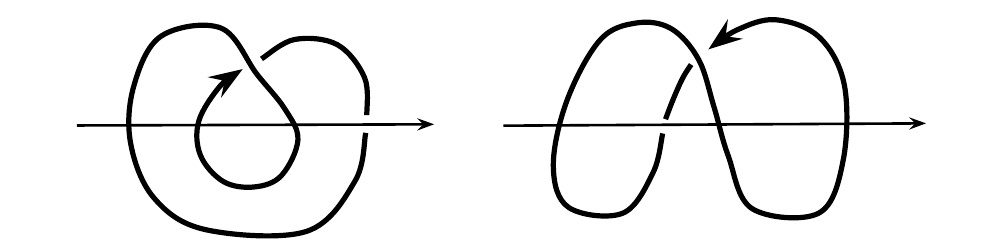}
\put(8,13.5){$a$}
\put(17.5,13.5){$b$}
\put(30,13.5){$c$}
\put(54,13.5){$a$}
\put(73,13.5){$b$}
\put(85.5,13.5){$c$}
\end{overpic}
\caption{These trisecants are \emph{reversed} (left) and {\em direct} (right) because the cyclic order of the points along $K$ is $acb$ and $abc$ respectively.}
\label{fig:trisecants} 
\end{center}
\end{figure}
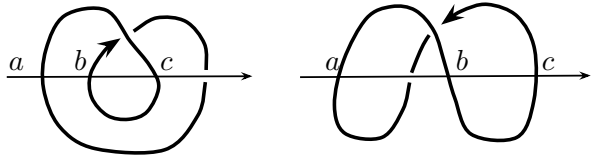

Just as with trisecants, we may compare the points of a quadrisecant line with their
ordering along the knot. For quadrisecant~$abcd$, the order along~$K$
is a cyclic order, and ignoring the orientation of~$K$ is just a
dihedral order. Thus there are~$\left|S_4/D_4\right|=3$ dihedral
orderings of a quadrisecant and non-oriented
knot. We can represent these equivalence classes by $abcd$, $abdc$ and $acbd$, where we have again chosen the
lexicographically least order as the name for each. Figure~\ref{fig:quadord}
illustrates these orderings.  

\begin{figure}
\begin{center}
\begin{overpic}{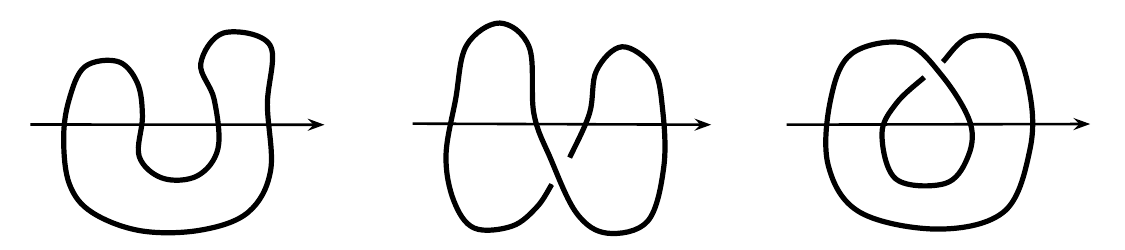}
\put(3.3,12){$a$}
\put(10,12){$b$}
\put(17,12){$c$}
\put(24.5,12){$d$}
\put(37.5,12){$a$}
\put(45,12){$b$}
\put(50.5,12){$c$}
\put(56.5,12){$d$}
\put(71,12){$a$}
\put(77,12){$b$}
\put(87,12){$c$}
\put(92.5,12){$d$}
\end{overpic}
\caption{From left to right, quadrisecants~$abcd$ are simple, flipped and alternating.}
\label{fig:quadord}
\end{center}
\end{figure}
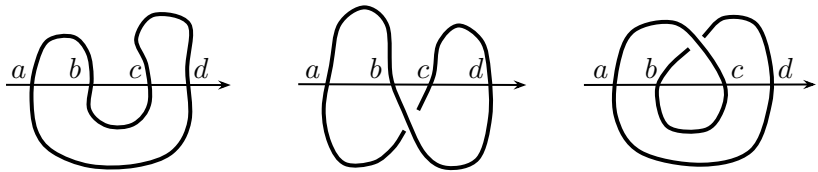

\begin{definition} Quadrisecants of type~$acbd$ are called \emph{alternating quadrisecants}. Quadrisecants of types~$abcd$ and~$abdc$ are called \emph{simple} and \emph{flipped} respectively.
\end{definition}

When discussing quadrisecant $abcd$, we will usually choose to orient $K$ so that $b\in\arc{ad}$. This means that the cyclic order of points along $K$ will be $abcd$, $abcd$, or $acbd$, depending on the type of quadrisecant.

%%%%%%%%%%%%%%%%%%%%%%%%%%%%%%

\subsection{Essential secants}
Before discussing known results about quadrisecants, we pause to introduce the notion of an essential secant.  (The definitions in this subsection can all be found in \cite{DDS}.)  G.~Kuperberg \cite{Kup} first introduced this idea (which he called ``topologically nontrivial'') in his paper about quadrisecants. Being essential captures part of the knottedness of the knot. 

Generically, the knot K together with the secant segment between two points on $K$ forms a knotted $\Theta$-graph in space (that is, a graph with three edges connecting the same two vertices).

\begin{definition}\label{def:esstheta}
  Suppose $\alpha$, $\beta$ and $\gamma$ are three disjoint simple
  arcs from points~$a$ to~$b$, forming a knotted $\Theta$-graph, as illustrated in Figure~\ref{fig:essdef}. 
  Let $X:=\R^3\setm (\alpha\cup\gamma)$, and let $\delta$ be a
  parallel curve to $\alpha\cup\beta$ in $X$. (By \emph{parallel} we mean
  that $\alpha\cup\beta$ and $\delta$ cobound an annulus embedded in
  $X$.) We choose $\delta$ to be homologically trivial in $X$
  (that is, so that $\delta$ has zero
  linking number with $\alpha\cup \gamma$). Let
  $h=h(\alpha,\beta,\gamma)\in\pi_1(X)$ denote the (free) homotopy class of
  $\delta$. Then $(\alpha,\beta,\gamma)$ is \emph{inessential} if
   $h$ is trivial.
We say that $(\alpha,\beta,\gamma)$ is \emph{essential} if it is not
inessential. 
\end{definition}

In other words, the ordered triple $(\alpha,\beta,\gamma)$ is \emph{inessential}
if there is a disk~$D$ bounded by $\alpha\cup\beta$ having no
  interior intersections with the knot $\alpha\cup\gamma$.
  (We allow self-intersections of~$D$, and interior intersections
  with $\beta$, as will be necessary if $\alpha\cup\beta$ is knotted.)

\begin{figure}
\centerline{\begin{overpic}{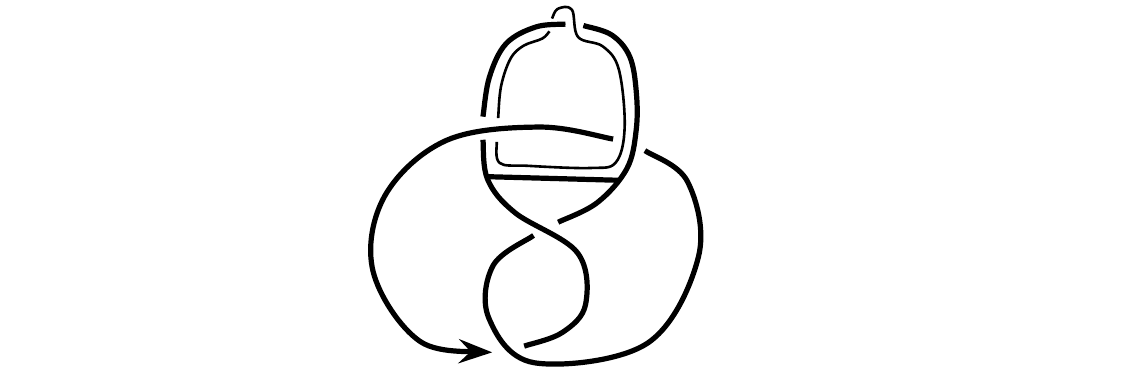}
\put(41,27){$\alpha$}
\put(41,17){$a$}
\put(56.5,16.5){$b$}
\put(48,14.5){$\beta$}
\put(31,6){$\gamma$}
\put(53,24){$\delta$}
\end{overpic}}
\caption{In the knotted
  $\Theta$-graph $\alpha\cup\beta\cup\gamma$, the ordered triple
  $(\alpha,\beta,\gamma)$ is essential. To see this, we find the parallel
  $\delta$ to $\alpha\cup\beta$ which has linking
  number zero with $\alpha\cup\gamma$ and note that is is
  homotopically nontrivial in the knot complement $\R^3\setm(\alpha
  \cup\gamma)$. In this illustration, $\beta$ is the straight segment
  $\overline{ab}$, so we may equally say that the arc $\alpha=\arc{ab}$ of the
  knot $\alpha\cup\gamma$ is essential.}
\label{fig:essdef}
\end{figure}
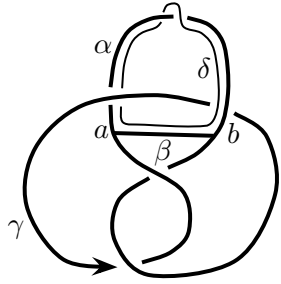

This notion is clearly a topological invariant of the (ambient isotopy)
class of the knotted $\Theta$-graph. 
We apply this definition to arcs of a knot $K$
below. But first, some useful notation. Let $a,b \in K$. The arc from
$a$ to~$b$ following the orientation of 
the knot is denoted $\arc{ab}$ and has length $\len{ab}$. The arc from $b$ to $a$, $\arc{ba}$, with length $\len{ba}$ is
similarly defined. The secant segment from $a$ to $b$ is denoted
$\overline{ab}$.

\begin{definition}\label{def:essarc}
 Assume $K$ is a nontrivial tame knot, $a,b\in K$, and $\ell=\overline{ab}$.
  We say $\arc{ab}$ is \emph{essential} if for every $\epsilon >0$ there exists some 
  $\epsilon$-perturbation of $\ell$ (with endpoints fixed) to a tame curve $\ell'$ such that   
  $K\cup \ell'$ forms an embedded $\Theta$ in which $(\arc{ab},\ell',\arc{ba})$ is
  essential.
\end{definition}

Note that this definition is quite flexible, as it allows for the situation where $K$ intersects~$\ell$. In \cite{DDS} Proposition 6.2, we show that it also ensures that the set of essential secants is closed in~$S$.

\begin{definition} 
  A secant $ab$ of $K$ is \emph{essential} if
  both subarcs $\arc{ab}$ and $\arc{ba}$ are essential. Otherwise it
  is \emph{inessential}. Let $ES$ be the \emph{set of essential secants
  in~$S$}. 
\end{definition}

It is straightforward to see that if $K$ is an unknot, then any arc $\arc{ab}$ is inessential. (Because the homology and homotopy groups of $X:=\R^3\setm K$ are equal for an unknot, so any curve having zero linking number with $K$ is homotopically trivial in $X$.) We can use Dehn's Lemma to prove a converse statement: if $a, b \in K$ and both
$\arc{ab}$ and $\arc{ba}$ are inessential, then $K$ is unknotted. (See for instance \cite{CKKS, DDS}.)

Later, in applications of these ideas, we need to find the least length of an essential arc, and use this to get better bounds on ropelength and distortion. This leads us to consider
what happens when arcs change from inessential to essential.

\begin{theorem}[\cite{DDS} Theorem 7.1] \label{thm:changeover} Suppose $\arc{ac}$ is in the boundary of the set of essential arcs for a knot $K$. (That is, $\arc{ac}$ is essential, but there are inessential arcs of $K$ with endpoints arbitrarily close to $a$ and $c$.) Then $K$ must intersect the interior of segment $\overline{ac}$, and in fact there is some essential trisecant $abc$.
\end{theorem}

Finally, to call a quadrisecant $abcd$ essential, we could follow Kuperberg and require
that the secants $ab$, $bc$ and $cd$ all be essential. But instead, we require this only of those secants whose endpoints are consecutive along the knot

\begin{definition} An $n$-secant $a_1a_2\dots a_n$ is essential if we have $a_ia_{i+1}$ essential for each $i$ such that one of the arcs $\arc{a_ia_{i+1}}$ and $\arc{a_{i+1}a_i}$ includes no other $a_j$.
\end{definition}

That is, for simple quadrisecants, all three secants
must be essential; for flipped quadrisecants the end secants $ab$ and $cd$ must be
essential; for alternating quadrisecants, the middle secant $bc$ must be essential.

%%%%%%%%%%%%%%%%%%%%%%%%%%%%%%

\subsection{Results about quadrisecants}

The simple dimension count outlined in Section~\ref{section:intro} means that we expect nontrivial tame knots to have quadrisecants. Indeed,
in 1933, E. Pannwitz \cite{Pann} first showed that every nontrivial
generic polygonal 
knot in~$\mathbb{R}^3$ has at least $2u^2$ quadrisecants, where
$u$ is the unknotting number of a knot\footnote{She actually stated her result in terms of the knottedness of $K$, which is the minimal number of singular points on the boundary of locally flat singular spanning disks of $K$, and is twice the unknotting number.}. (It is entirely possible that Heinz Hopf suggested this problem to her. In the endnote to his paper, I. F\'ary \cite{Fary} mentions that Hopf used quadrisecants to prove that knots have total curvature greater than or equal to $4\pi$.)

 In the early 1980's, Morton and Mond \cite{MM} rediscovered Pannwitz' result. They independently proved every nontrivial generic knot
has a quadrisecant, and they conjectured that a generic 
knot with crossing number~$n$ has at least~${n}\choose{2}$
quadrisecants. It was not until 1994 that Kuperberg \cite{Kup} managed
to extend the result and show that {\em all} (nontrivial tame) knots in
$\mathbb{R}^3$ have a quadrisecant. To do this, he introduced the notion of an essential secant. In 1998, C. Schmitz \cite{Schm} nearly proved that alternating quadrisecants exist for nontrivial tame knots in Hadamard manifolds, but in his proof some quadrisecants may degenerate to trisecants.

In 2004, the following result was proved in my PhD thesis.
\begin{theorem} [\cite{Denne}]\label{thm-quad} All nontrivial tame knots have at least one alternating quadrisecant.  All nontrivial knots of finite total curvature have at least one essential alternating quadrisecant. 
\end{theorem}
Also in 2004, R. Budney, J. Conant, K. Scannell and D. Sinha~\cite{BCSS} gave a geometric interpretation of the second order Vassiliev invariant. To do this, they used the techniques of compactified configuration spaces and Goodwillie calculus to show that this invariant can be computed by counting alternating quadrisecants with appropriate multiplicity (both for long knots and closed knots). This result implies the existence of alternating quadrisecants for many knots, while Theorem~\ref{thm-quad} shows existence for all nontrivial tame knots.
 
Later in 2013, Budney's Master's student G. Flowers \cite{Flowers} gave a natural extension of that formula to closed knots, but instead of counting quadrisecants lines, he counts five- and six-point cocircuarities. The relative ordering of the knot and intersecting circle are again central to the arguments, and the idea of an alternating quadrisecant becomes that of either a `satanic' or `thelemic' circle.

In 2007, M. Sommer developed a wonderful jReality application {\em Visualization in Geometric Knot Theory} \cite{VGKT-jReality} as part of his Diploma Thesis \cite{VGKT-thesis}. This program allows the user to view the set of trisecants for polygonal knots, and even see how it changes as a vertex is moved. The program also shows the different projections of this set onto different planes, from which the user can read off the different quadrisecants.

In 2008, T.~Fiedler and V.~Kurlin \cite{FK-fiber} viewed quadrisecants in a different way. They fixed a straight line in $\R^3$ and a fibration around it by half-planes, and studied knots in general position with respect to this fibration using secants and quadrisecants lying in fibers. They give the  minimum number of all fiber quadrisecants and fiber extreme secants which occur during an isotopy from one knot to another in terms of a sum of unordered and coordinated writhes of a particular kind of  projection.

In 2009, J.~Viro \cite{Viro} estimated from below the number of lines meeting each of four disjoint smooth curves in both $\R P^3$ and $\R^3$. In~$\R^3$ her arguments may be translated to quadrisecant lines, however her count involves linking numbers and appears to be different from the one in \cite{BCSS}. 

%%%%%%%%%%%%%%%%%%%%%%%%%%%%%
%%%%%%%%%%%%%%%%%%%%%%%%%%%%%

\subsection{Counting quadrisecants and quadrisecant approximations.}

For a generic nontrivial knot, we expect that there are a finite number of quadrisecants. (Indeed, this was one of the results of \cite{BCSS}.) The open and very challenging problem is to give bounds on the number of quadrisecants for a given knot type.  Many of the papers cited in the previous section relate the number of quadrisecants to some topological invariant of knots (unknotting number, finite type invariants, writhe), or give conjectures about what this might be. However none of the known results seem to be close to the number of observed quadrisecants for knots with large crossing number.  Indeed, there are not even conjectures about asymptotic bounds for the number of quadrisecants for particular families of knots. 

There have been several papers which have grappled with this question for knots with small crossings. In 2005, G.T.~Jin \cite{Jin-quad}, gave trigonometric parametrizations of the $3_1, 4_1, 5_1$ and $5_2$ knots and found all quadrisecants for these parametrizations. He also defined the quadrisecant approximation of a knot $K$. 
\begin{definition} Let $K$ be a knot with finitely many quadrisecants intersecting $K$ in finitely many points. These points divide $K$ into a finite number of subarcs. Replacing the subarcs by straight lines gives a polygonal closed curve called the {\em quadrisecant approximation} of $K$.
\end{definition}
Jin then conjectured the following.
\begin{conjecture}[\cite{Jin-quad}] \label{quad-approx} If a knot $K$ has finitely many quadrisecants, then the quadrisecant approximation $\hat{K}$ has the knot type of $K$. Furthermore, $K$ and $\hat{K}$ have the same set of quadrisecants.
\end{conjecture}

In 2011, Jin and S. Park \cite{JP} prove that every hexagonal (6-sided) trefoil knot has exactly three alternating quadrisecants. They then prove that the quadrisecant approximation of a hexagonal trefoil knot is also a trefoil knot, and moreover has the same three alternating quadrisecants as the original knot.

Most recently, S. Bai, C. Wang and J. Wang \cite{BWW} proved that the  Conjecture~\ref{quad-approx} is false. They give two examples of an unknot: the first has a quadrisecant approximation which is not even an embedded curve, and the second has a quadrisecant approximation which is a left-handed trefoil knot. They then use a particular type of connect sum operation to show that there is a polygonal knot in every knot type whose quadrisecant approximation is either not embedded, or contains a trefoil summand. 

There are still many interesting questions to explore here. For example, it is entirely possible that Conjecture~\ref{quad-approx} is true for polygonal knots which have minimum stick number, or for ideal knots (those which minimize some kind of energy like ropelength).

Returning to the question of counting quadrisecants, A. Cruz-Cota and T. Ramirez-Rosas \cite{CC-RR} recently proved the first result giving an upper (rather than a lower) bound on the number of quadrisecants.

\begin{theorem}[\cite{CC-RR}] Let $K$ be a polygonal knot in general position, with exactly $n$ edges. Then $K$ has at most $\displaystyle \frac{n}{12}(n-3)(n-4)(n-5)$ generic quadrisecants.
\end{theorem}

We finish this section by noting that the results of \cite{BCSS} are the only ones relating a count of {\em alternating} quadrisecants to a knot invariant. We have work-in-progress \cite{Den-quad} 
showing that there is at least $u(K)$ number of alternating quadrisecants of a knot $K$,  where $u(K)$ is the unknotting number of $K$. Examples from \cite{Jin-quad} have lead us to conjecture the following (say for knots of finite total curvature).
\begin{conjecture} The figure-8 knot has at least one essential flipped quadrisecant. The $5_2$ knot has at least one essential simple quadrisecant.
\end{conjecture}

%%%%%%%%%%%%%%%%%%%%%%%%%%%%%%
%%%%%%%%%%%%%%%%%%%%%%%%%%%%%%%

\section{Key ideas in showing quadrisecants exist}\label{section:key-ideas}

Observe that quadrisecants are formed when several trisecants share common
points. Quadrisecant~$abcd$ includes four trisecants (1)$abc$,
(2)$abd$, (3)$acd$, (4)$bcd$.  Pannwitz \cite{Pann} showed
quadrisecants exist by looking for pairs of trisecants like~(1)$abc$
and~(3)$acd$, where the first and third points of trisecant~$abc$ are
the same as the first and second points of trisecant~$acd$. Kuperberg \cite{Kup} showed that quadrisecants exist by looking for pairs
of trisecants like~(2)$abd$ and~(3)$acd$, where the first and third
points of the trisecants are the same. We proved  \cite{Denne} that alternating quadrisecants exist by looking at families~(1)$abc$ and~(2)$abd$, where the first and second points of the trisecants are the same. The quadrisecant count in \cite{BCSS} boiled down to looking at linking numbers of these sets of trisecants in the 3rd associahedron (Stasheff polytope) --- a compactified configuration space. 

All of these arguments presuppose that nontrivial knots have trisecants. Fortunately this is easy to prove, as was originally shown by Pannwitz.

\begin{lemma}[\cite{Pann}] \label{lem:tri1}  Each point of a nontrivial tame knot~$K$ is the 
first point of at least one trisecant. 
\end{lemma}
\begin{proof} Suppose there is a point~$a\in K$ which is not
the start point of any trisecant. The union of all
chords~$\overline{ab}$ for~$b\in K$ is a disk with boundary~$K$. If two chords~$ab$ and~$ac$ intersect at a place other than~$a$, then they overlap
and one is a subinterval of another. They form a trisecant ($abc$ or~$acb$), contrary to the assumption. Thus the disk is embedded and $K$ unknotted, a contradiction.
\end{proof}

The structure of set of trisecants (and quadrisecants) can be understood more easily when we restrict the class of knots considered, such that the restricted class is dense in the set of tame knots. All of the previous work on  quadrisecants we have discussed does this, exactly how depends on the context chosen. We choose to work with polygonal knots with some extra assumptions.
\begin{definition} \label{def:generic} We say that the polygonal knot $K$ in $\R^3$ is {\em generic} if the following conditions are satisfied
\begin{itemize}
\item No four vertices of $K$ are coplanar, no three vertices of $K$ are collinear.
\item Given three pairwise skew edges of $K$, no other edge of $K$ is contained in the quadric generated by those edges.
\item There are no $n$-secants for $n\geq 5$.
\end{itemize}
\end{definition}

Once the existence of {\em essential} quadrisecants is established for generic knots, a limit argument (see \cite{Kup}) is used to show that any nontrivial tame knot has a quadrisecant. The key point here is that the limit of an essential secant remains a secant --- it does not degenerate in the limit.

%%%%%%%%%%%%%%%%%%%%%%%%%%%%%%

\subsection{Trisecants and quadrisecants.} 

We now discuss how trisecants and quadrisecants can arise in generic polygonal knots. The first condition in Definition~\ref{def:generic} means that three (or more) adjacent edges cannot be coplanar. Suppose $e_i$ and $e_{i+1}$ are adjacent edges. A 1-parameter family of trisecants arises when a third edge $e_j$ of the knot intersects certain parts of the plane spanned by $e_i$ and $e_{i+1}$. This family is either homeomorphic to $[0,1]$ or $[0,1)$ depending on which region $e_j$ intersects (see Figure~\ref{fig:planar-trisecants}).  A quadrisecant is formed when a fourth edge intersects one of the trisecant lines. Genericity implies that two non-adjacent edges cannot be coplanar, thus there can be at most one quadrisecant in this case.

\begin{figure}[htbp]
\begin{center}
\begin{overpic}[scale=0.9]{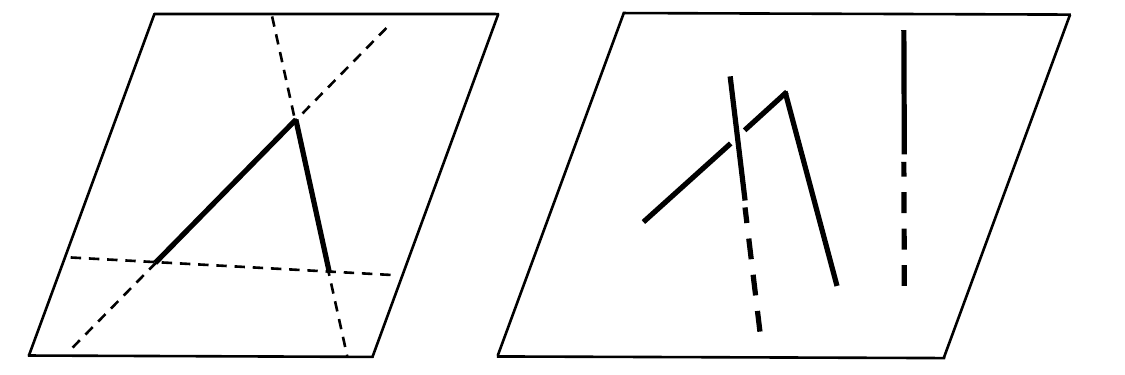}
\put(16,16){$e_i$}
\put(29,15){$e_{i+1}$}
\put(16,25){$(*)$}
\put(5,6){$(*)$}
\put(21,13){$(*)$}
\put(30, 6){$(*)$}
\put(34,25){$(*)$}
\put(58,19){$e_i$}
\put(73,16){$e_{i+1}$}
\put(61,25){$e_j$}
\put(76,28){$e_{k}$}
\end{overpic}
\caption{The planes are spanned by edges $e_i,e_{i+1}$, On the left, there are trisecant lines intersecting $e_i, e_{i+1}$ if the third edge intersects any one of the five regions marked (*), otherwise there are no trisecants. On the right, the family of trisecants intersecting $e_i,e_{i+1}, e_j$ is homeomorphic to $[0,1]$, while the family intersecting $e_i,e_{i+1}, e_k$ is homeomorphic to $[0,1)$.}
\label{fig:planar-trisecants}
\end{center}
\end{figure}
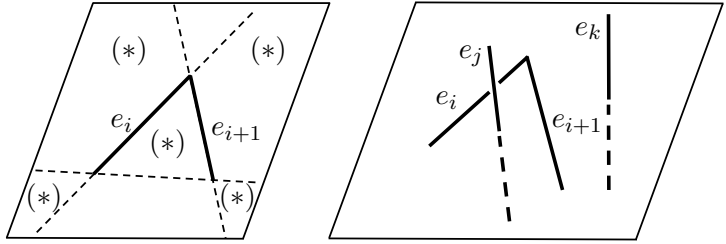

Before moving on to the next case, we pause to remind the reader about some well known facts about {\em doubly ruled surfaces} (see for instance \cite{HC-V, PW, Otal}). A triple of pairwise skew lines $l_1, l_2, l_3$ determines a unique quadric, a doubly-ruled surface $H$ (see Figure~\ref{fig:doubly-ruled}). This is either a hyperbolic paraboloid, if the three lines are parallel to one plane, or a hyperboloid of one sheet, otherwise. Each point of $H$ lies on a unique line from each ruling. The lines  $l_1, l_2, l_3$ belong to one of the rulings of the surface, and every line intersecting all three lines belongs to the other. Thus a fourth line $l_4$ intersecting $H$ yields two (or one) quadrisecant line(s) intersecting $l_1, l_2, l_3$, and $l_4$. There are an infinite number of quadrisecants only when $l_4$ is contained in $H$.  

\begin{figure}[htbp]
\begin{center}
\begin{overpic}[scale=0.3]{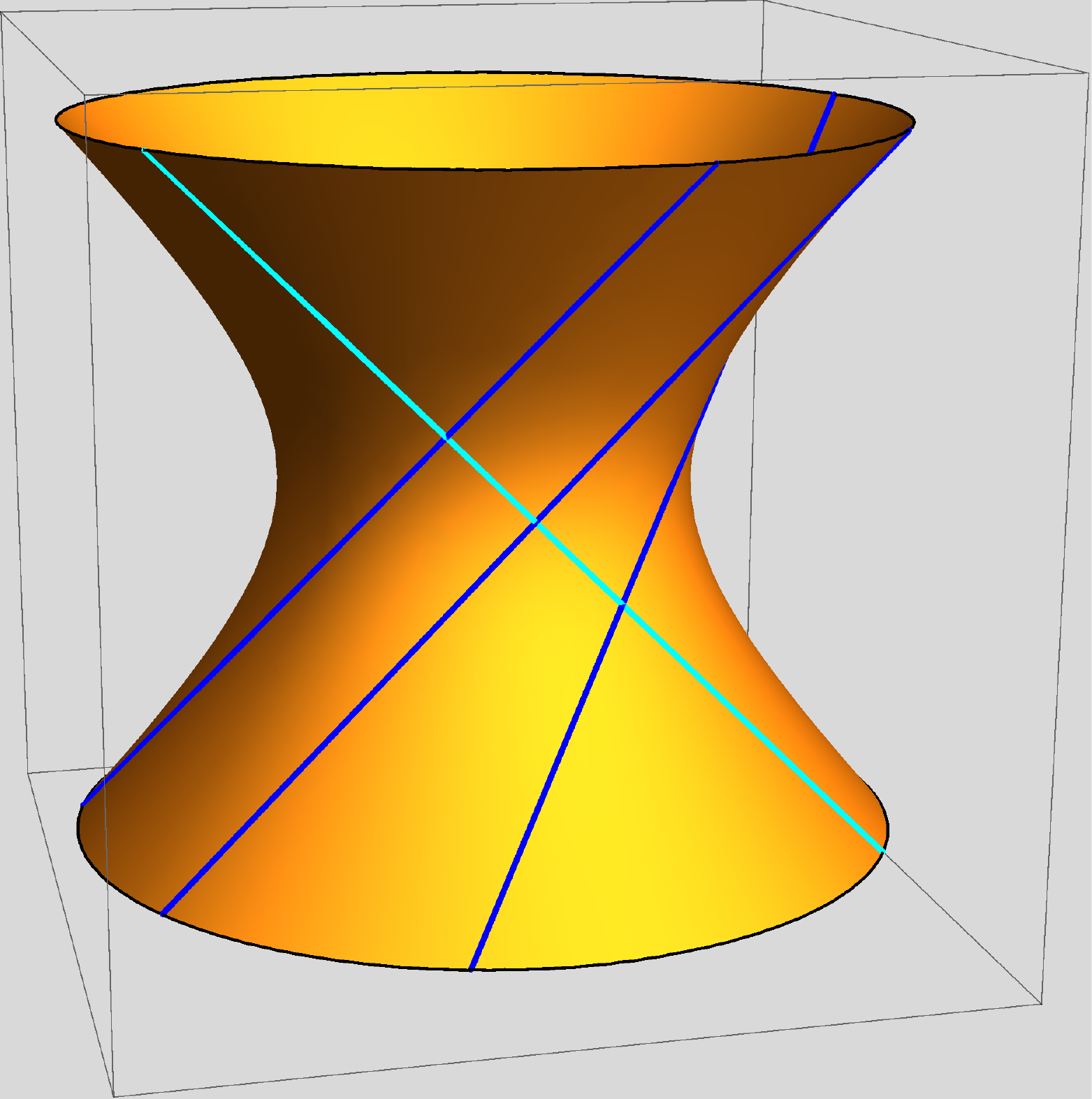}
\end{overpic}
\ \ \ 
\begin{overpic}[scale=0.3]{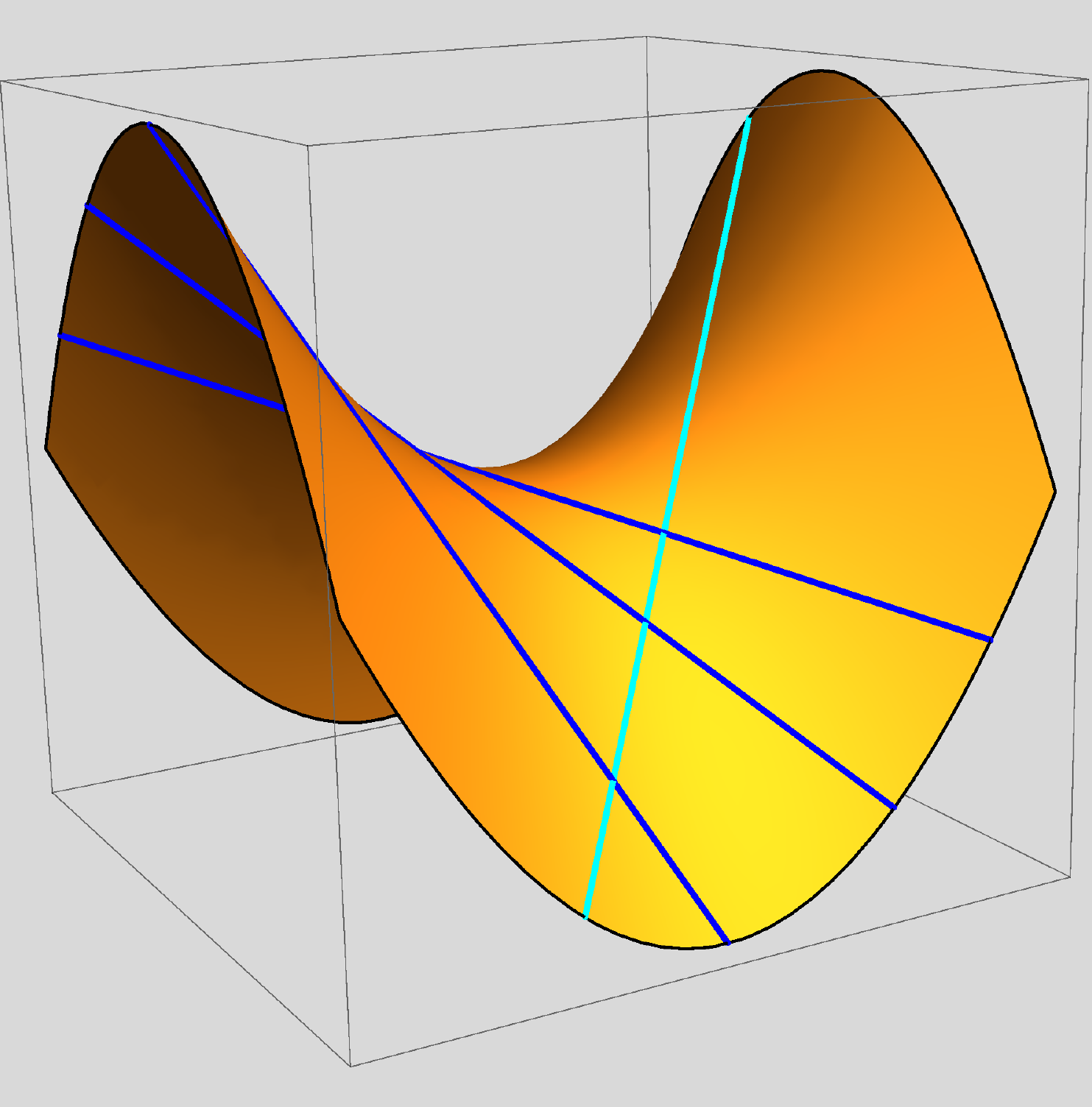}
\end{overpic}
\caption{Doubly ruled surfaces produced in {\em Mathematica}. On the left a hyperboloid of one sheet, on the right a hyperbolic paraboloid. The three blue lines lie in one ruling, the cyan line in the other.}
\label{fig:doubly-ruled}
\end{center}
\end{figure}

Returning our attention to generic polygonal knots, we see that a pairwise skew triple of edges $e_i,e_j,e_k$ either has 0 or 1 trisecant, or has a 1-parameter family of trisecants. A quadrisecant is formed when a fourth edge intersects the quadric in a trisecant line. There are 1 or 2 quadrisecants in this case, since the second condition in Definition~\ref{def:generic} means that there cannot be an infinite number of quadrisecants.

%%%%%%%%%%%%%%%%%%%%%%%%%%%%%%

\subsection{Structure of the set of trisecants.} 

Using the ideas from above, we can piece together the structure of the set of trisecants $\mathcal{T}\subset K^3$. Indeed this approach has been used to understand (and compute) $\mathcal{T}$ and the set of quadrisecants in many papers (such as \cite{BWW, CC-RR, Denne, JP, VGKT-thesis}). We expect that generically, $\mathcal{T}$ is a 1-manifold.
 Of course, we can add more (generic) conditions to Definition~\ref{def:generic} to give even more control over $\mathcal{T}$ (for example those in \cite{Denne, VGKT-thesis}). We omit these details here, and also the many technical but straightforward details needed to prove the following two results.

\begin{proposition}[c.f. \cite{Denne, VGKT-thesis}]\label{prop:embedded}
Let $K$ be a nontrivial generic polygonal tame knot. In $K^3$,
$\overline{\mathcal{T}}$ is a compact 1-manifold with boundary,
embedded in $K^3$ in a piecewise smooth way with $\mathcal{T}\subset
K^3 \setm \tilde{\Delta}$ and
$\partial{\mathcal{T}}\subset{\Delta}$. Moreover, each component of $\mathcal{T}$ is either a simple closed curve or a simple open arc.
\end{proposition}

\begin{proposition}[c.f. \cite{Denne}]\label{prop:allproj}
Let $K$ be a nontrivial generic polygonal tame knot. 
In $K^2$, the projection $\pi_{ij}$ ($i<j$ and
$i,j=1,2,3$) is a piecewise smooth immersion of $\mathcal{T}$ into the set of secants $S$,
and $T=\pi_{ij}(\mathcal{T})$ intersects itself (transversally) at
double points. 
\end{proposition}

The astute reader will have realized that we really need to work with {\em essential} trisecants if we are to find {\em essential} (alternating) quadrisecants. Recall that the set of essential secants $ES$ is a subset of $S$.  In order to find essential alternating quadrisecants, we restrict our attention to the trisecants $abc$ which are essential in the second segment $bc$.  Thus $\mathcal{ET}=\pi^{-1}_{23}(ES)\cap  \mathcal{T}$, and we define $\mathcal{ET}^d$ and $\mathcal{ET}^r$ to be the sets of essential trisecants of direct and reversed orderings in $K^3$. Since we are interested in finding essential trisecants which share the first two points, we project  $\mathcal{ET}$ to $S$ in a particular way.

\begin{definition}[\cite{Denne}]  Let $ET=\pi_{12}(\mathcal{ET})$ be the projection of the set of
  essential trisecants to the set of secants $S$ and similarly define
  $ET^d:=\pi_{12}(\mathcal{ET}^d)$ and $ET^r:=\pi_{12}(\mathcal{ET}^r)$.  
\end{definition}

We can easily prove that a version of Proposition~\ref{prop:allproj} holds for $ET$: For a nontrivial generic polygonal tame knot, $ET$ is a piecewise immersed 1-manifold which intersects itself transversally at double points (\cite{Denne}). Finally, we observe that (essential) alternating quadrisecants occur precisely when direct and reversed (essential) trisecants intersect in $S$.

\begin{lemma}[\cite{Denne}] \label{lem:essquadkey} Let $ab\in ET^d\cap ET^r$ in $S$. This
  means that there exists $c$ and~$d$ such that $abc\in \mathcal{ET}^r$ and
  $abd\in \mathcal{ET}^d$. Then either $abcd$ or $abdc$ is an essential
  alternating quadrisecant. \qed
\end{lemma}

The proof of the existence of alternating quadrisecants, Theorem~\ref{thm-quad}, follows the same pattern as the proof of Lemma~\ref{lem:tri1}. We proceed by contradiction, and assume that  $ET^d\cap ET^r=\emptyset$ in $S$. We then show that $K$ has to be unknotted, a contradiction.
The details of proving $K$ unknotted are quite involved and so we omit them here. [However, very briefly,  since $K$ is a polygonal knot, we can show that $ET^d$, respectively $ET^r$, stay at least the minimum edge length away from the bottom, respectively top, boundary of $S$. Since $ET^d$ and $ET^r$ do not intersect, we use a Meyer-Vietoris argument to construct a loop winding once around $S$ which avoids the set of essential trisecants. We then use this loop to find a spanning disk whose boundary is the knot. After further technicalities, the loop theorem is invoked to show this disk is embedded.]

%%%%%%%%%%%%%%%%%%%%%%%%%%%%%%
%%%%%%%%%%%%%%%%%%%%%%%%%%%%%%%

\section{Applications of essential secants and quadrisecants} \label{section:applications}

%The existence of essential secants and (essential alternating) quadrisecants gives new insights into the geometry of tame knots.

\subsection{Total curvature}

For smooth closed curves, the total curvature can be thought of as the
total angle through which the unit tangent vector turns (or the length
of the tangent indicatrix). J. Milnor \cite{Mil} defined the {\em total curvature} $\kappa(\gamma)$ of an arbitrary closed curve $\gamma$ to be the supremal total curvature of inscribed polygons (where the total curvature is the sum of the exterior angles), and showed the two definitions are equivalent.  In 1929, W. Fenchel \cite{Fen} proved that the total curvature of a closed curve in $\R^3$ is greater than or equal
to $2\pi$, equality holding only for plane convex curves. In 1947, K. Borsuk \cite{Borsuk}
extended this result to $\R^n$ and conjectured the following.

\begin{theorem}
A nontrivial tame knot in $\R^3$ has total curvature greater than $4\pi$.
\end{theorem}

This result was first proved around 1949 by both Milnor \cite{Mil} and F\'ary \cite{Fary}.
It has since become known as the F\'ary-Milnor theorem.  In \cite{CKKS}, the theorem was proved using the existence of second hull for a knotted
curve (defined below). In~1998, the F\'ary-Milnor theorem was independently extended to knotted curves in Hadamard\footnote
{A Hadamard manifold is a complete simply-connected
Riemannian manifold with non-positive sectional curvature.}
manifolds by Schmitz \cite{Schm} and S. Alexander and R. Bishop \cite{AB}.  Interestingly, Alexander and Bishop show that a doubly covered bi-gon is inscribed in a polygonal curve inscribed in the nontrivial tame knot, an idea somewhat close to the midsegment of an alternating quadrisecant.

A new proof of the F\'ary-Milnor theorem is given by the existence of alternating quadrisecants for any nontrivial tame knot $K$. An alternating quadrisecant can be thought of as an inscribed polygon in $K$, and has total curvature $4\pi$. Thus by definition, $\kappa(K)\geq 4\pi$. To get a strict inequality, simply observe that a knot is not coplanar. 

More recently, H. Gerlach, P. Reiter and H. von der Mosel \cite{GRM} have written about {\em elastic knots}, which are limit configurations of energy minimizers of an energy consisting of the classic bending energy and a small multiple of ropelength (defined below). One of the many results in this paper is an extension of
the classic F\'ary-Milnor theorem on total curvature to the $C^1$-closure of the knot class. The proof of this result also relies on the existence of alternating quadrisecants.

%%%%%%%%%%%%%%%%%%%%%
\subsection{Second Hull}

The convex hull of a connected set $K$ in $\R^3$ is characterized by the fact that
every plane through a point in the hull must intersect $K$. If $K$ is a closed curve,
then a generic plane must intersect $K$ an even number of times. Thus every plane
through each point of the convex hull is cut by $K$ at least twice. In proving the total curvature result above, Milnor observed that for a nontrivial tame knot, there are planes in every direction which cut the knot four times. More generally, there
are points through which {\em every} plane cuts the knots four times. This idea was formalized by J. Cantarella, Kuperberg, R. Kusner and J.M. Sullivan \cite{CKKS}, where the authors defined these points as the {\em second hull} of a knot. 

\begin{definition}[\cite{CKKS}] Let $K$ be a closed curve in $R^3$. Its $n$th hull $h_n(K)$ is the set of points $p\in \R^3$ such that $K$ cuts every plane $P$ through $p$ at least $2n$-times.
\end{definition}

Cantarella \emph{et al.} proved that the second hull of
a nontrivial tame knot in $\mathbb{R}^3$ is nonempty. This paper conjectured
the existence of alternating quadrisecants for nontrivial tame knots in
$\mathbb{R}^3$ as another way of proving that the second hull is
nonempty. This is because the mid-segment $bc$ of alternating
quadrisecant $abcd$ is in the second hull.

%%%%%%%%%%%%%%%%%%%%%%%%%%%%%%%%
\subsection{Ropelength}

The ropelength problem asks to minimize the length of a
knotted curve subject to maintaining an embedded tube of fixed diameter
around the tube; this is a mathematical model of tying the knot tight
in a rope of fixed thickness.

More technically, the {\em thickness} $\tau(K)$ of a space curve $K$ is defined \cite{GM} to be
twice the infimal radius $r(a, b, c)$ of circles through any three distinct points of $K$. It
is known \cite{CKS} that $\tau(K) = 0$ unless $K$ is $C^{1,1}$ (meaning that its tangent direction
is a Lipschitz function of arclength). When $K$ is $C^1$, we can define normal tubes
around $K$, and then indeed $\tau(K)$ is the supremal diameter of such a tube that
remains embedded. We note that in the existing literature thickness is sometimes
defined to be the radius rather than diameter of this thick tube. Following others lead, we define the ropelength as follows.
\begin{definition} The {\em ropelength}  of a knot $K$ is the (scale-invariant) quotient of length over thickness, $Rop(K)=len(K)/\tau(K)$.
\end{definition}

Cantarella, Kusner and Sullivan \cite{CKS} proved that any (tame) knot or link type has a ropelength minimizer and gave certain lower bounds for the ropelength of links; these are sharp in certain simple cases where each component of the link is planar.  In this section we outline our joint work~\cite{DDS}, showing how essential alternating quadrisecants are used to prove that nontrivial knots have ropelength at least~\altval.  This is an improvement on the bound of 12 from \cite{Diao-24} and is greater than the conjectured bound of 15.25 from \cite{CKS}.

%%%%%%%%%%%%%%%%%%%%%%%
\subsubsection{Ropelength basics}

Because the ropelength problem is scale invariant, we find it most convenient
to rescale any knot $K$ to have thickness (at least) $1$. This implies that $K$ is a $C^{1,1}$
curve with curvature bounded above by $2$.  For any point $a\in\R^3$, let $B(a)$ denote the open unit ball centered at $a$.  We now give several well known results about the local structure of thick knots (see for instance \cite{Diao, CKS} and \cite{DDS}). Most proofs are elementary and we omit them here.

\begin{lemma} \label{lem:rope-basics} Let $K$ be a knot of unit thickness. If $a\in K$, then $B(a)$ contains
a single unknotted arc of $K$; this arc has length at most $\pi$ and is transverse to the
nested spheres centered at $a$. If $ab$ is a secant of $K$ with $|a-b| < 1$, then the ball of
diameter $ab$ intersects $K$ in a single unknotted arc (either $\arc{ab}$ or $\arc{ba}$) whose length is at most $\arcsin |a - b|$. \qed
\end{lemma}

As an immediate corollary, we see that if $K$ has unit thickness, $a, b \in K$ and $p \in\arc{ab}$ with $a, b \notin B(p)$, then the complementary arc $\arc{ba}$ lies outside $B(p)$.

\begin{lemma}\label{lem:project} Let $K$ be a knot of unit thickness. If $a\in K$, then the radial
projection of $K\setm\{a\}$ to the unit sphere $\partial B(a)$ does not decrease length.   \qed
\end{lemma}

\begin{lemma} \label{lem:len-tri}Suppose $K$ has unit thickness, and $p, a, b \in K$ with $p \notin \arc{ab}$. Let $\angle apb$ be the angle between the vectors $a - p$ and $b - p$. Then $\len{ab} \geq \angle apb$. In particular, if $apb$ is a reversed trisecant in $K$, then $\len{ab}\geq \pi$.  \qed
\end{lemma}

%This last lemma immediately gives simple lower bounds on ropelength for any knot with a quadrisecant. Simple, flipped and alternating quadrisecants have different numbers of reversed trisecants. The ropelength of a knot with a simple, flipped or alternating
%quadrisecant is at least $\pi$, $2\pi$ or $3\pi$, respectively.

Given a thick knot K with quadrisecant abcd, we can bound its ropelength in
terms of the distances along the quadrisecant line. Whenever we discuss such a
quadrisecant, we will abbreviate these three distances as $r := |a - b|$, $s := |b - c|$
and $t := |c - d|$. We start with some lower bounds for $r$, $s$ and $t$ for alternating quadrisecants. 

\begin{lemma} [\cite{DDS} Lemma 4.2] If $abcd$ is an alternating quadrisecant for a knot of unit thickness, then $r\geq 1$ and $t \geq 1$. With the usual orientation, the entire arc $\arc{da}$ thus lies outside $B(b)\cup B(c)$. If $s\geq 1$ as well, then $\arc{ac}$ lies outside $B(b)$ and $\arc{bd}$ lies outside $B(c)$.  \qed
\end{lemma}

As suggested by the discussion above, we will often find ourselves in the situation
where we have an arc of a knot known to stay outside a unit ball. We
can compute exactly the minimum length of such an arc in terms of the following
functions.
\begin{definition} For $r\geq 1$, let $f(r) := \sqrt{r^2 - 1} + \arcsin(1/r)$. For $r, s \geq 1$ and
$\theta \in [0, \pi]$, the minimum length function is defined by
$$m(r, s, \theta) =
\begin{cases} \sqrt{r^2 + s^2 - 2rs \cos\theta} &  \text{ if $\theta\leq \arccos(1/r) + \arccos(1/s)$} \\
f(r) + f(s) + (\theta - \pi)  & \text{if $\theta \geq \arccos(1/r) + \arccos(1/s)$}
\end{cases}
$$
\end{definition}

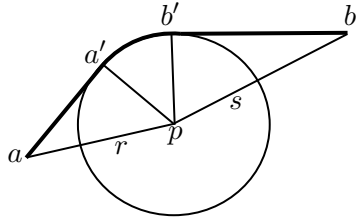
\begin{figure}
\begin{center}
\begin{overpic}{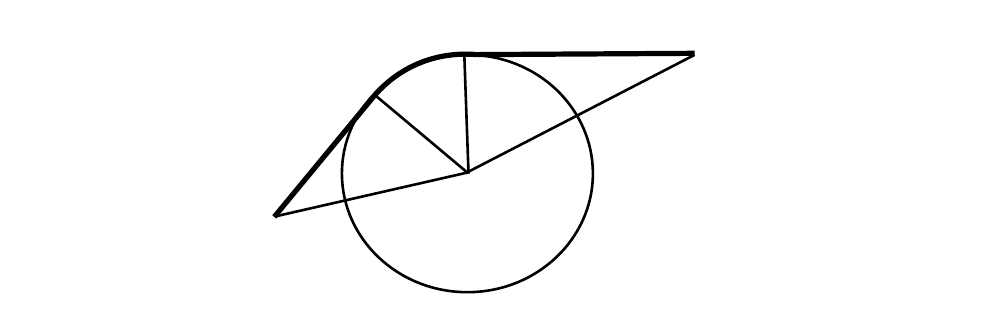}
\put(25,9){$a$}
\put(35,22){$a'$}
\put(45,27){$b'$}
\put(69,27){$b$}
\put(46,12){$p$}
\put(39,10){$r$}
\put(54,16){$s$}
\end{overpic}
\caption{The shortest arc from $a$ to $b$  avoiding the ball $B(p)$ consists of straight segments and an arc of the ball.}
\label{fig:dist-fn}
\end{center}
\end{figure}

The function $f(r)$ will arise again in other situations. The function $m$ was
defined exactly to make the following bound sharp:

\begin{lemma} \label{lem-lengthball} Any arc $\gamma$ from $a$ to $b$ staying outside $B(p)$ has length at least $m(|a  - p|, |b - p|,\angle apb)$.  \qed
\end{lemma}

An important special case is when $\theta = \pi$.  If $a$ and $b$ lie at distances $r$ and $s$ along opposite rays from $p$
(so that $\angle apb = \pi$) then the length of any arc from $a$ to $b$ avoiding $B(p)$ is at least
$$
f(r) + f(s) =\sqrt{r^2 - 1} + \arcsin(1/r) + \sqrt{s^2 - 1} + \arcsin(1/s).
$$
%%%%%%%%%%%%
\subsubsection{Length of essential secants and quadrisecants}

We will improve our previous ropelength bounds by getting bounds on the
length of an essential arc. Intuitively, we expect an essential arc $\arc{ab}$ of a knot to ``wrap at least halfway around'' some point on the complementary arc $\arc{ba}$. Although when $|a - b| = 2$ we
can have $\len{ab} = \pi$, when $|a - b| < 2$ we expect a better lower bound for $\len{ab}$. Even though in fact an essential $\arc{ab}$ might instead ``wrap around'' some point on itself, we can still derive the desired bound.

\begin{lemma}[\cite{DDS}] If $\arc{ab}$ is an essential arc in a knot $K$ of unit thickness, then $|a-b|\geq 1$ and $\len{ab}\geq g(|a-b|)$, where
$$g(|a-b|)= \begin{cases} 2\pi - 2\arcsin(|a-b|/2) & \text{ if $0 \leq |a-b| \leq 2$}, \\
\pi & \text{ if $|a-b|\geq 2$}.
\end{cases}$$
\end{lemma}

\begin{proof} (Sketch)
If $|a-b|<1$, then by Lemma~\ref{lem:rope-basics}  the ball of diameter $\overline{ab}$ contains a single unknotted arc of $K$, and thus inessential.

Knowing that sufficiently short arcs starting at any point $a$ are inessential, consider the shortest arc $\arc{aq}$ which is essential. From Theorem~\ref{thm:changeover} there is a trisecant $apq$ with both secants $ap$ and $pq$ essential, thus $a$ and $q$ are outside $B(p)$. Since $ap$ is essential, $apq$ is reversed and by Lemma~\ref{lem:len-tri} we get $\len{ab}\geq\len{aq}\geq \pi$.

Note that $|a-b|\in[1,2]$, so $2\pi-\arcsin(|a-b|/2) \leq 5\pi/3$. Considering again $\arc{aq}$ with reversed trisecant $apq$, we have $b\notin\arc{aq}$ and $\len{aq}\geq \pi$, so we may assume $\len{qb}\leq 2\pi/3$ or the bound is trivially satisfied.

The remainder of the proof involves finding a bound on $\len{qb}$ using the fact that $\arc{qp}$ is essential. There are two cases. The first is where $b\notin B(p)$ and the whole arc $\arc{aqb}$ stays outside $B(p)$. Thus $\len{ab}$ is greater than the length of the radial projection of the arc  onto $\bdry B(p)$, giving the result. The second case is where $b\in B(p)$. Here, we let $\arc{qy}$ be the shortest essential arc starting at $q$ and use a short argument (omitted) to show that $\len{ab}\geq 5\pi/3$.
\end{proof}

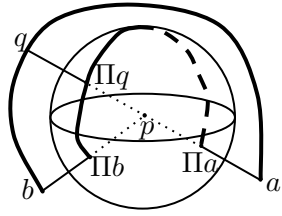
\begin{figure}
\begin{center}
\begin{overpic}{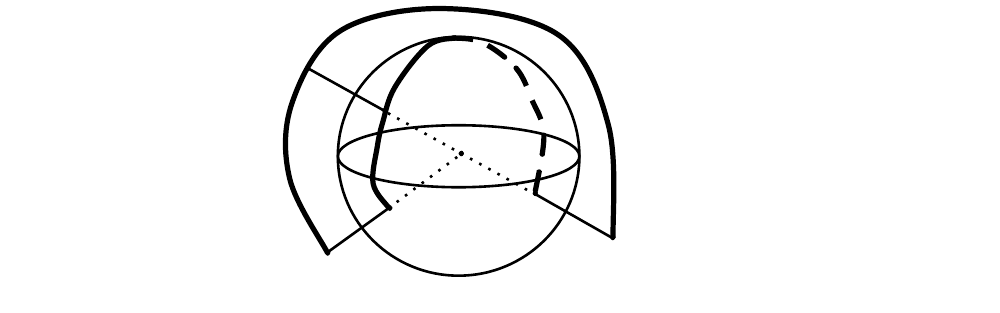}
\put(62,6){$a$}
\put(51,8.5){$\Pi a$}
\put(30,5){$b$}
\put(39,8){$\Pi b$}
\put(45.5,13.5){$p$}
\put(29,25){$q$}
\put(39.5,20){$\Pi q$}
\end{overpic}
\caption{The projection of $\arc{ab}$ to the ball $B(p)$ does not increase the length $\len{ab}$.}
\label{fig:ess-sec}
\end{center}
\end{figure}

It is now straightforward to compute a lower bound of the ropelength of a unit thickness knot $K$ with an essential quadrisecant $abcd$.  There are three cases, depending on the type of quadrisecant. 

\begin{itemize}
\item[1.] When $K$ has an essential simple quadrisecant, it has length $\len{ab} + \len{bc} + \len{cd} + \len{da}\geq (g(r)+f(r)) + (g(s)+s)+(g(t)+f(t))$. By minimizing each term separately, we find that $K$ has ropelength at least $10\pi/3 + 2\sqrt{3} +2 > 15.936$.
\item[2.] When $K$ has an essential flipped quadrisecant, it has length $\len{ab} + \len{bd} + \len{dc} + \len{ca}\geq (g(r)+f(r)) + 2f(s) + (g(t)+f(t))$. By minimizing each term separately, we find that $K$ has ropelength at least $10\pi/3 + 2\sqrt{3} > 13.936$.
\item[3.] When $K$ has an essential alternating quadrisecant, the length of $K$ is $\len{ac}+\len{cb}+\len{bd}+\len{da}\geq 2f(r) + (2f(s)+g(s)+s) + 2f(t)$. Again, by minimizing each term separately, we find $K$ has ropelength at least 15.66. 
\end{itemize}
Together with Theorem~\ref{thm-quad}, we conclude the following.
\begin{theorem}[\cite{DDS}] Any nontrivial knot has ropelength at least 15.66
\end{theorem}

Several independent numerical simulations (see for instance \cite{Pie,Sul})
have found a trefoil knot with ropelength less than~$16.374$, which is presumably close to the minimizer. This is about 5\% greater than our bound, so there is not much room for improvement, although a careful analysis based on tangent directions at $b$ and $c$ could yield a slightly better bound.

As a final remark, we note that the ropelength problem is still open for all knot and most link types --- it is a rich source of open questions. For example, there are many results relating ropelength to other knot invariants, in particular to crossing number, for instance \cite{BS, CFM-rope, DE-hamilton, DE-curv-rope-cr, DEZ-conway, HHKNO}. More recently, there have been several papers giving a set of necessary and sufficient conditions for ropelength criticality, for example \cite{ SvdM, CFKSW, SW, CFKS-rl}.

%%%%%%%%%%%%%%%%%%%%%%%%%
\subsection{Distortion}
M. Gromov introduced the notion of distortion for curves (see for instance \cite{Grom-dist,Grom-dil, Grom-83}).
\begin{definition}
If $\gamma$ is a rectifiable curve in $\R^3$, then its distortion is defined to be the quantity:
$$\delta(\gamma) = \sup_{p,q\in\gamma}\frac{d_\gamma(p, q)}{d_{\R^3}(p, q)} \geq 1.$$
where $d_\gamma$ denotes the arclength along $\gamma$ and $d_{\R^3}$ denotes the Euclidean distance in $\R^3$.
\end{definition}

Gromov showed that for any simple closed curve $\gamma$, we
have $\delta(\gamma) \geq \frac{1}{2}\pi$, with equality if and only if $\gamma$ is a circle, thus determining $\delta(\text{unknot})$. He then asked whether every knot type can be built with say, $\delta\leq 100$. 
In  \cite{DS-dist}, we proved that for any nontrivial tame knot we have $\delta(K) \geq \frac{5}{3}\pi$. To do this, we first showed that {\em any} nontrivial tame knot has a shortest essential secant.  Then, borderline-essential arcs and their lengths (namely Theorem~\ref{thm:changeover} and Lemma~\ref{lem-lengthball}) were key tools used in our distortion computations. Our bound is of course not sharp, but numerical simulations  \cite{Mull-PhD} have found a trefoil knot with distortion
less than 7.16, so we are not too far off. We expect the true minimum distortion for a trefoil is closer to that upper bound than to our lower bound. 

In 2011, J. Pardon \cite{Pardon} proved that the distortion of $T_{p,q}$, a $(p,q)$ torus knot, is $\delta(T_{p,q}) \geq \frac{1}{160}\min(p,q)$. This shows the answer to Gromov's question is no. His main theorem considers isotopy classes of simple loops in a piecewise-linear embedded surface of genus $g\geq 1$, and his inequality involves certain minimum geometric intersection numbers of these loops.

In general, distortion is quite tricky to get a handle on. In \cite{DS-dist} we give an example of a wild knot, the connect sum of infinitely many trefoils, with distortion less than 10.7. Unlike the ropelength case, having finite distortion does not put any regularity conditions on the curve. Because of this, it is an open problem to establish the existence of minimizers of $\delta$ on any nontrivial knot class. In 2007, C. Mullikin \cite{Mull} started to develop a calculus of variations theory for distortion. However, it is not clear whether the techniques developed for ropelength criticality can be easily applied to this situation.

%%%%%%%%%%%%%%%%%%%%%%%%%
\subsection{Final Remarks}
Quadrisecants have made an appearance in other parts of knot theory as well. For, example they give a starting place to finding information about two ``super-invariants'' of knots.
 
Recall that the bridge index of a knot is defined to be the minimum number of bridges in all possible diagrams of a knot in its knot type. The {\em superbridge index} was first defined by N. Kuiper \cite{Kuip}, and is 
$$ sb[K] = \min_{K'\in[K]} \max_{\vec{v}\in S^2} b_{\vec{v}}(K'),$$ 
where $b_{\vec{v}}(K)$ is the number of bridges (or local maxima) of an orthogonal projection $K\rightarrow \R\vec{v}$. Kuiper then computed the superbridge index for all torus knots. Later, Jin and C.B. Jeon 
 \cite{Jin-sb,JJ-sb}) used quadrisecants to show that there are only finitely many knot types with superbridge index 3.
 
The supercrossing index of a knot is related to crossing index in the same way that superbridge index is related to bridge index. Namely, the {\em supercrossing index} of $K$ is
$$ sc[K]=\min_{K\in[K]} \max_{\vec{v}\in S^2} (\#\text{ of crossings}).$$
More simply, we maximize the number of crossings we see generated by a given conformation of a knot, and then minimize over all conformations.  This invariant was studied by C. Adams {\em et al.} \cite{Adams-supercrossing} and was related to other knot invariants like stick index. The existence of quadrisecants mean that every nontrivial tame knot has supercrossing index at least 6. To see this, simply perturb the projection along the quadrisecant line to find six crossings.

In summary, quadrisecants and essential secants appear in many different parts of knot theory. They are a useful tool in understanding many phenomena of knots and links, as they form a bridge between topological and geometric properties. For all this, quadrisecants are still not completely understood. One of the more important open problems is to give bounds on the number of quadrisecants for knot and link types.

%%%%%%%%%%%%%%%%%%%%%%%%%%%%%%

\section{Acknowledgement}
I first worked on quadrisecants as part of my PhD thesis, and I am forever grateful to my advisor John M. Sullivan for introducing me to them. I am also deeply appreciative of the many helpful conversations I have had over the years with Stephanie Alexander, Ryan Budney, Jason Cantarella, and Cliff Taubes about this material. Thanks also go to my coauthor Yuanan Diao for first realizing that quadrisecants could be applied to the ropelength problem.

%%%%%%%%%%%%%%%%%%%%%%%%%%%%%%

\bibliographystyle{amsalpha}
\bibliography{bibliography.bib}

\end{document}